\documentclass[12pt]{article}
\usepackage{amssymb}

\usepackage{amsmath,amsfonts,amssymb}

\newcommand{\M}{{\mathbb M}}
\newcommand{\R}{{\mathbb R}}
\newcommand{\N}{{\mathbb N}}

\newcommand{\I}{{\mathbb I}}

\renewcommand{\~}{\widetilde}

\newcommand{\hatP}{{\widehat{\cal P}}}
\newcommand{\hatF}{{\widehat{\cal F}}}

\usepackage{bbm}

\def\B{{\cal B}}
\def\P{{\cal P}}

\def\H{{\cal H}}
\def\D{{\cal D}}
\def\F{{\cal F}}
\def\V{{\cal V}}
\def\A{{\cal A}}

\def\e{{\cal E}}

\def\e{\mbox{${\cal E }$}}
\def\fk{\mbox{${(F_k)_{k \in \N}}$}}

\newtheorem{defn}{Definition}[section]
\newtheorem{theo}[defn]{Theorem}

\newtheorem{prop}[defn]{Proposition}

\newtheorem{rem}[defn]{Remark}

\newenvironment{proof}{{\bf Proof }}{{\vskip 0.1cm \hfill$\Box$}}

\begin{document}
\centerline{\Huge \sf Weak existence of the squared Bessel and }
\vspace{0.5cm}
\centerline{\Huge \sf CIR process with skew reflection on a }
\vspace{0.5cm}
\centerline{\Huge \sf deterministic time dependent curve}
\vspace{1cm}
\centerline{\large \sf Gerald  TRUTNAU}
\vspace{1cm}
\hspace{-0.65cm}Department of Mathematical Sciences,
Seoul National University,
San56-1 Shinrim-dong Kwanak-gu,
Seoul 151-747, South Korea (e-mail: trutnau@snu.ac.kr)\\ \\
{\bf Summary}: Let $\sigma>0,\delta\ge 1, b\ge 0$, $0<p<1$. Let $\lambda$ be a continuous and positive 
function in $H^{1,2}_{loc}(\R^+)$. 
Using the technique of moving domains (see \cite{Tr4}), and classical direct 
stochastic calculus, we construct for positive initial conditions a pair of continuous and positive semimartingales $(R,\sqrt{R})$ with
$$
dR_t=\sigma\sqrt{R_t}dW_t +\frac{\sigma^2}{4}(\delta-bR_t)dt + (2p-1)d\ell^0_t(R-\lambda^2),
$$ 
and 
\begin{eqnarray*}
d\sqrt{R}_t&=&\frac{\sigma}{2}dW_t+
\frac{\sigma^2}{8}\left (\frac{\delta-1}{\sqrt{R}_t}-b\sqrt{R}_t\right )dt
+(2p-1)d\ell^0_t(\sqrt{R}-\lambda)\\ 
&&\hspace*{+4cm}+\frac{\mathbbm{I}_{\{\delta=1\}}}{2}\ell_t^{0+}(\sqrt{R}),
\end{eqnarray*} 
where the symmetric local times $\ell^0(R-\lambda^2), \ell^0(\sqrt{R}-\lambda)$, 
of the respective semimartingales $R-\lambda^2, \sqrt{R}-\lambda$ are related through the formula
$$
2\sqrt{R}d\ell^0(\sqrt{R}-\lambda)=d\ell^0(R-\lambda^2).
$$ 
Well-known special cases are the (squared) 
Bessel processes (choose $\sigma=2$, $b=0$, and 
$\lambda^2\equiv 0$, or equivalently $p=\frac12$), and the Cox-Ingersoll-Ross process 
(i.e. $R$, with $\lambda^2\equiv 0$, or equivalently $p=\frac12$).
The case $0<\delta<1$ can also be handled, but is different. If $|p|>1$, then there is no solution.
Boundary value problems for second-order, parabolic equations;  Finance, portfolios, investment;
Secondary: Dirichlet spaces, Potentials and capacities. \\  
\section{Introduction}
Consider a continuous, and positive function $\lambda:\R^+\to \R^+$, $\lambda\in H^{1,2}_{loc}(\R^+)$.  
For parameters $\sigma>0,\delta, b\ge 0$, $p\in ]0,1[$, and positive initial conditions, we construct the Cox-Ingersoll-Ross process with skew reflection on $\lambda^2$ (here $\lambda^2(t)=\lambda(t)\cdot\lambda(t)$), 
i.e. a weak solution to
\begin{eqnarray}\label{skewbesq}
dR_t=\sigma\sqrt{|R_t|}dW_t+\frac{\sigma^2}{4}(\delta-bR_t)dt+ (2p-1)d\ell^0_t(R-\lambda^2),
\end{eqnarray} 
where $\ell_t^0(R-\lambda^2)$ is the symmetric local time 
at zero of the continuous  semimartingale $R-\lambda^2$. The equation is in particular complicated because of the singular diffusion coefficient in combination with a parabolic local time. 
For example, the first naive idea to put $S_t=R_t-\lambda^2(t)$ in order to simplify (\ref{skewbesq}) doesn't improve at all the situation. One ends up with 
$$
dS_t=\sigma\sqrt{|S_t+\lambda^2(t)|}dW_t+\frac{\sigma^2}{4}(\delta-bS_t-b\lambda^2(t)+\frac{8}{\sigma^2}\lambda(t)\lambda'(t))dt+ (2p-1)d\ell^0_t(S),
$$
i.e. an equation with degenerate time dependent diffusion coefficient and still remaining reflection term. For equations of this type general existence and uniqueness results are completely unknown. Playing around with equation 
(\ref{skewbesq}) one quickly observes that the sole employment of direct stochastic calculus doesn't seem to be helpful for the construction of such highly singular processes. 
In fact, constructing a solution to (\ref{skewbesq}) means to construct a solution to the Cauchy problem related to the generator of the CIR process with skew boundary conditions on $\lambda^2$. This is in view of the coefficients a complicated task. Our approach here is to provide a general method for the construction of these type of diffusions applying a recently developed technique of moving domains (see \cite{Tr4}) 
in combination with classical direct stochastic calculus. Let us explain this method in detail. We keep $\delta\ge 1$ in order to simplify (for the case $\delta\in(0,1)$ see Remark \ref{extension}).
First one can readily check that (\ref{skewbesq}) is the square of a solution $Y,Y_0\ge 0$ to
\begin{eqnarray}\label{skewbes}
dY_t=\frac{\sigma}{2}dW_t+\frac{\sigma^2}{8}\left (\frac{\delta-1}{Y_t}-bY_t\right )dt+
(2p-1)\mathbbm{I}_{\{\lambda>0\}}d\ell^0_t(Y-\lambda)+
\frac{\mathbbm{I}_{\{\delta=1\}}}{2}d\ell_t^{0+}(Y),
\end{eqnarray} 
where $\ell_t^{0+}$ the upper (right-continuous) local time. In fact, just apply It\^o's formula to see that $R_t:=Y_t^2$, $R_0\ge 0$, solves  
$$
dR_t=\sigma\sqrt{R_t}dW_t +\frac{\sigma^2}{4}(\delta-bR_t)dt + (2p-1)2\sqrt{R_t}d\ell^0_t(\sqrt{R}-\lambda),
$$ 
and then notice that\\
\begin{eqnarray}\label{relloc}
2\sqrt{R_t}d\ell^0_t(\sqrt{R}-\lambda)=d\ell^0_t(R-\lambda^2).
\end{eqnarray} 
The relation (\ref{relloc}) can be shown by probabilistic means 
using a product formula for local times (see \cite{Yan}, and \cite{ouk}). In fact, we have  
$R-\lambda^2=(\sqrt{R}+\lambda)(\sqrt{R}-\lambda$), but see also Remark \ref{addcalc}(i) for an independent analytic proof. 
So far we explained why it is enough to construct (\ref{skewbes}) in order to get (\ref{skewbesq}). The advantage of (\ref{skewbes}) in comparison with (\ref{skewbesq}) is that it is 
much better suited to be obtained by a change of measure since there is no coefficient in the martingale part. \\ 
The next step is to decompose $\lambda$ into $\lambda=\beta+\gamma$, where 
$\beta,\gamma\in H^{1,2}_{loc}(\R^+)$ are continuous representatives, with
$\beta$ decreasing,  and $\gamma$ increasing. Consider the following moving domain
$$
E:=\{(t,x)\in \R^+\times\R|x\ge -\gamma(t)\}.
$$
Since $-\gamma(t)$ decreases, $E_t:=\{x\in \R| (t,x)\in E\}=[-\gamma(t),\infty)$ is 
increasing in $t$. As reference measure on $E$ consider 
$$
m(dxdt):=\rho(t,x)dxdt,
$$ 
where 
$$
\rho(t,x):=\left ((1-p) \mathbbm{I}_{[-\gamma(t)),\beta(t))}(x)+
p \mathbbm{I}_{[\beta(t),\infty)}(x)\right )
|x+\gamma(t)|^{\delta-1}e^{-\frac{bx^2}{2}}
$$
is assumed to be increasing in $t$, that is 
\begin{eqnarray*}
\rho(s,x)\le \rho(t,x) \ \ \ \forall \ 0\le s\le t,\  x\in E_s.
\end{eqnarray*}
Due to the monotonicity assumptions on $\rho$ (and $E$) we are in the theoretical framework of \cite{Tr4}. 
Note that the monotonicity  assumptions are in fact restricting as these do not allow all imaginable choices of parameters. 
We will comment on this below in the introduction. 
Applying the theory of time-dependent Dirichlet forms on monotonely moving domains of \cite{Tr4} we deduce that the 
process associated to the time-dependent Dirichlet form (\ref{form}) solves  
\begin{eqnarray}\label{transform}
X_t & = & X_0+\frac{\sigma}{2}B_t+\frac{\sigma^2}{8}\int_0^t\frac{\delta-1}{X_s+\gamma(s)}-bX_s ds
\nonumber\\ 
&&+(2p-1)\int_0^t\I_{\{\beta(s)>-\gamma(s)\}}d\ell_s^0(X-\beta) 
+\frac{\I_{\{\delta= 1\}}}{2}\ell_t^{0+}(X+\gamma),\\ \nonumber 
\end{eqnarray}
(that is (\ref{rootCIR})). 
The construction is carried out in section \ref{constr} with advanced and heavy machinery, i.e. the {\it Theory of Generalized Dirichlet Forms}. 
We couldn't see how else to do it. In fact there is a lot of work done with direct stochastic calculus in the case of \lq\lq nice\rq\rq 
coefficients and non-parabolic local time 
(see e.g. \cite{lg} as showcase), or also in the Brownian motion case (no coefficients) with parabolic local time  (see e.g. \cite{bcs}, or \cite{Wein}). 
So, the type of equation (\ref{skewbesq}) seems to be a real novelty. Now let us come back to the explanation of its construction. 
In section \ref{GDF1} we first derive the existence 
of a diffusion associated to the time-dependent Dirichlet form (\ref{form}) (see Theorem \ref{Hunt}). 
In order to obtain finally (\ref{transform}) we first identify (\ref{id4}) as the process associated to the time-dependent Dirichlet form (\ref{form}). 
Since in general Dirichlet form theory the drift is 
associated to a signed smooth measure via the Revuz correspondence (\ref{I3}), we can see Markovian local times, i.e. singular increasing processes 
that are associated to smooth measures and their potentials, 
in the equation (\ref{id4}) (cf. explanation at the beginning of section \ref{GDF3}). That these Markovian local times exist is shown in Proposition \ref{smoo}. 
In particular the signed smooth measure corresponding to the 
Markovian local times in (\ref{id4}) is given in (\ref{smoo1}) with $F$ equal to the identity (see also (\ref{id1}), (\ref{id2})). The Markovian 
local times are indeed semimartingale local times. This is deduced in section \ref{GDF3} by comparing Fukushima's decomposition for generalized Dirichlet forms (\ref{I5}) with 
Tanaka's formula (\ref{tanaka}) (see explanation right after Remark \ref{constloc}). The final formulas in order to obtain (\ref{transform}) are then given in (\ref{loc2}), and (\ref{loc3}). \\
Once a solution to (\ref{transform}) is constructed we just define 
$$
W_t:=B_t+\frac{1}{4\sigma }\int_0^t 8\gamma'(s)+\sigma^2 b\gamma(s)ds,
$$ 
and then under a change of measure with density
$$
e^{-\frac{1}{4\sigma }\int_0^t 8\gamma'(s)+\sigma^2 b\gamma(s)dB_s-
\frac{1}{32\sigma^2}\int_0^t |8\gamma'(s)+\sigma^2 b\gamma(s)|^2 ds}
$$
we see that $Y_t:=X_t+\gamma(t)$ is a solution to (\ref{skewbes}), 
hence $R:=Y^2$ solves (\ref{skewbesq}) and $R$ is constructed. \\ \\
Let us now comment on the monotonicity assumptions on $\rho$ (and $E$). The time dependent Dirichlet form (\ref{form}) consists of two parts. The first is a symmetric bilinear form and the second is a 
perturbation with the time derivative. The time derivative is generated by the semigroup $(U_t)_{t\ge 0}$ of (\ref{SG}). $U_t$ pushes the support of a function $F\in C^1_0(E)$ by $t$ to the left. 
If we want to guarantee that  $(U_t)_{t\ge 0}$ becomes a contraction semigroup on $L^2(E,\rho dxdt)$ then this is only possible if the support of $U_tF$ is 
smaller than the one of $F$ (i.e. $E_t$ increases in $t$), and if $\rho(s-t,x)$ is smaller than $\rho(s,x)$ ($\rho$ increases in $t$). The semigroups corresponding to the 
symmetric part of the time dependent Dirichlet form (\ref{form}) and to the time derivative are nearly \lq\lq orthogonal\rq\rq. Therefore in order to guarantee that the generator corresponding 
to the whole form  (\ref{form}) satisfies the positive maximum principle (cf. e.g. \cite[\mbox{chapter 4}]{ek}) we need $(U_t)_{t\ge 0}$ to be a contraction. 
Various considerations lead to the belief that this might be difficult if not impossible to improve. In any case this is an intriguing question. 
For example it is not possible to add some killing with a constant $c$ to make $(e^{-ct}U_t)_{t\ge 0}$ a contraction (and then to remove afterwards the killing). 
If  $p\in (\frac12,1)$, so that $1-p<p$, then $\rho(\cdot,x)$ always increases, or if 
$\beta=const$, then $\rho(\cdot,x)$ increases for any $p\in (0,1)$. Therefore our construction takes place for arbitrary 
$\lambda\in H^{1,2}_{loc}(\R^+)$, if $2p-1>0$, or for increasing $\lambda$, if $2p-1<0$ (cf. below (\ref{mono}), and Remark \ref{constloc}, \ref{addcalc}(ii) for more general $\lambda$). 
It is remarkable that pathwise uniqueness in \cite{Tr6} could be deduced for $2p-1<0$ if $(\lambda^2)'\ge \frac{\sigma^2}{4}(\delta-b\lambda^2)$. Thus for $p<\frac{1}{2}$ and some strictly decreasing 
$\lambda$ but which is still bounded below by the mean-reverting level $\frac{\delta}{b}$. We also find remarkable that we were not able to construct a solution in the extreme cases $p=0$ and $p=1$.\\ 
In Remark \ref{whysym}(i) we show that if $|p|>1$, then there is no solution to (\ref{skewbesq}). 
The construction of $R,\sqrt{R}$ is in the sense of equivalence of additive functionals of Markov processes (see Remark \ref{quasi}). 
The case $0\le \delta<1$ can also be handled, but is different, see Remark \ref{extension}. In \cite{Tr6} it is shown that a solution to (\ref{skewbesq}) always 
stays positive when started with positive initial condition. One can hence discard the absolute value under the square root in (\ref{skewbesq}) 
as in the classical situation where $\lambda^2\equiv 0$. Nota bene, that in case $\lambda^2\equiv 0$, the local time $\ell(R)$ vanishes as a direct consequence of the occupation time formula.
Let us also remark that the parallel work \cite{Tr6} is completely different from this work as well as from the subject as from the used techniques 
since there exclusively probabilistic arguments are employed.\\
We really think that our approach is a novelty, and that equations (\ref{skewbesq}), and (\ref{skewbes}) are truly worth to be studied. 
The parallel work \cite{Tr6} which deals with pathwise uniqueness of (\ref{skewbesq}) shows that stochastic calculus  is possible, even with such highly singular equations. 
It encourages to go further. For instance it could be challenging to further investigate ergodic behavior of (\ref{skewbesq}) under the assumption $\lim_{t\to\infty}\lambda^2(t)=const.$ 
Or, is it possible to write down the distribution of R and/or to calculate explicitly the relevant corresponding quantities (mean, variance, etc.)?
Finally we think that equation (\ref{skewbesq}) can also have suitable interpretation in terms of multi-factor term structure models 
(i.e. models corresponding to a system of SDEs of basic univariate models) if we 
manage to provide enough analytic tractability and a sufficiently general framework to express no-arbitrage in the corresponding model. \\ \\
\section{Construction of the skew reflected CIR process}\label{constr}\medskip
\subsection{Construction of a solution to (\ref{transform})}\label{GDF} \medskip
\subsubsection{The Generalized Dirichlet Form associated to a solution of (\ref{transform})}\label{GDF1} \medskip
Throughout this article $\mathbbm{I}_A$ will denote the indicator function of a set $A$. 
Let $E:=\R^+\times\R^+$, where $\R^+:=\{x\in \R|\,x\ge 0\}$. Let 
$C_0^{1}(\R^+)=\{f:\R^+\to \R|\,\exists u\in C_0^{1}(\R)\mbox{ with } u\mathbbm{I}_{\R^+}=f\}$, 
and $C_0^{1}(\R)$ denotes the continuously differentiable functions with compact support in $\R$.
Let $H^{1,2}(\R^+)$ be the Sobolev space of order one in $L^2(\R^+)$, that is the completion of 
$C_0^{1}(\R^+)$ w.r.t. $|\phi|_{H^{1,2}(\R)}=
(\int_{\R^+}|\partial_u \phi|^2+|\phi|^2 du)^{\frac12}$. When considering an element of $H^{1,2}(\R^+)$, 
we always assume that is it continuous by choosing such a version. 
Later we will use the notions $\partial_u,du$ for the space variable 
(notation $u=x$) as well as for the time variable (notation $u=t$, or $u=s$). For the space-time variable we use $y$, e.g. $y=(s,x)$, $y=(t,x)$.
Let $H^{1,2}_{loc}(\R^+)$ denote the space of all continuous $\phi:\R^+\to\R$ such that 
$\phi f\in H^{1,2}(\R^+)$  for any $f\in C_0^{1}(\R^+)$. \\ \\
If $\lambda\in H^{1,2}_{loc}(\R^+)$, then it has a uniquely determined continuous version w.r.t. 
the Lebesgue measure $dt$. We will always assume that $\lambda$ is continuous. Furthermore we assume that 
 $\lambda$ is positive, i.e. $\lambda\ge 0$. 
In particular $\lambda=\beta+\gamma$, where 
$\beta,\gamma\in H^{1,2}_{loc}(\R^+)$, and 
$\beta$ is decreasing,  $\gamma$ is increasing. 
Indeed, since $\partial_t \lambda \in L^2_{loc}(\R^+)$ we may consider its positive part 
$(\partial_t \lambda)^+$, and  its negative part 
$(\partial_t \lambda)^-$, and fix from now on  
\begin{eqnarray}\label{decomp}
\beta(t):=-\int_0^t (\partial_t \lambda)^-(s)ds+\lambda(0), \ \ \ \ \ \ \gamma(t):=
\int_0^t (\partial_t \lambda)^+(s)ds.
\end{eqnarray}
$\lambda=\beta+\gamma\ge 0$ implies $\beta\ge -\gamma$. Consider the following moving domain
$$
E:=\{(t,x)\in \R^+\times\R|x\ge -\gamma(t)\}.
$$
Observe, that its $t$-section $E_t=\{x\in \R| (t,x)\in E\}=[-\gamma(t),\infty)$ is 
increasing in $t$ since $-\gamma(t)$ decreases in $t$. In particular 
$E=\cup_{t\ge 0}\{t\}\times E_t$.\\
Let $\delta\ge 1$, $b\in \R^+$ (for the case $\delta\in(0,1)$ see Remark \ref{extension}). As reference measure on $E$ we take 
$$
m(dy)=m(dxdt):=\rho(t,x)dxdt,
$$ 
where 
$$
\rho(t,x):=\left ((1-p) \mathbbm{I}_{[-\gamma(t)),\beta(t))}(x)+
p \mathbbm{I}_{[\beta(t),\infty)}(x)\right )
|x+\gamma(t)|^{\delta-1}e^{-\frac{bx^2}{2}}
$$
is assumed to be increasing in $t$, that is 
\begin{eqnarray}\label{mono}
\rho(s,x)\le \rho(t,x) \ \ \ \forall \ 0\le s\le t,\  x\in E_s.
\end{eqnarray}
For instance, if  $p\in (\frac12,1)$, so that $1-p<p$, then $\rho(\cdot,x)$ always increases, or if 
$\beta=const$, then $\rho(\cdot,x)$ increases for any $p\in (0,1)$.\\ 
Due to the monotonicity properties of $\rho$ (and $E$) we are in the framework of \cite{Tr4}. 
More precisely, there is a time-dependent 
generalized Dirichlet form $\e$ with domain $\F\times \V\cup \V\times\hat\F$ on 
$\H:=L^2(E,m)$ 
which we determine right below. 
For $q\ge 1$ let 
$$
C_0^{q}(E):=\{f:E\to \R|\,
\exists u\in C_0^{q}(\R^2)\mbox{ with } u\mathbbm{I}_{E}=f\},
$$ 
and $C_0^{q}(\R^2)$ denotes the $q$-times continuously differentiable functions with compact support 
in $\R^2$.
Let $0<\sigma\in \R$
\begin{eqnarray}\label{DF}
\A(F,G):=\frac{\sigma^2}{8}\int_{0}^{\infty}\int_{-\gamma(s)}^{\infty}\partial_x F(s,x) \,\partial_x G(s,x) \rho(s,x) dxds;\ \ \ F,G\in C_0^{1}(E), 
\end{eqnarray}
with closure $(\A,\V)$ in $\H$. The closability easily follows since $\rho$ satisfies a Hamza type condition 
(see \cite[\mbox{Lemma 1.1}]{Tr4}). Let $\A_{\alpha}(F,G):=\A(F,F)+\alpha(F,F)$, $\alpha>0$, 
where $(\cdot,\cdot)$ is the inner product in $\H$. 
For $K\subset E$ compact, the capacity related to $\A$ is defined by
\begin{eqnarray}\label{capa}
\mbox{Cap}^{\A}(K)=\inf\{\A_1(F,F);F\in C^{1}_{0,K}(E)\}, 
\end{eqnarray}
where $C^{1}_{0,K}(E)=\{F\in {C^1_0(E)}| F(s,x)\ge 1, \forall (s,x)\in K\}$. For general 
$A\subset E$ it is extended by inner regularity. Define  
\begin{eqnarray}\label{SG}
U_t F(s,x):=F(s+t,x);\ \ \ F \in C_0^{1}(E).
\end{eqnarray}
It then follows from results in \cite{Tr4} that $(U_t)_{t\ge 0}$ can be extended to a 
$C_0$-semigroup of contractions on $\H$ which can be restricted to a 
$C_0$-semigroup on $\V$. For the corresponding generator 
 $(\partial_t, D(\partial_t,\H)$ on $\H$ it follows that 
$$
\partial_t:D(\partial_t,\H\cap \V)\to \V'
$$
is closable as operator from $\V$ to its dual $\V'$ (see \cite[\mbox{I.Lemma 2.3.}]{St1}). 
Let $(\partial_t,\F)$ be the closure.   
$\F$ is a real Hilbert space with norm  
$$
|F|_{\F}:=\sqrt{|F|_{\V}^2+|\partial_t F|_{\V'}^2}.
$$ 
The adjoint semigroup $(\widehat{U}_t)_{t\ge 0}$ of $(U_t)_{t\ge 0}$ in $\H$ 
can be extended to a $C_0$-semigroup on $\V'$. The corresponding generator 
$(\hat\Lambda,D(\hat\Lambda,\V'))$ is the dual operator of $(\partial_t,D(\partial_t,\V))$. 
$\widehat{\F}:=D(\hat\Lambda,\V')\cap\V$ is a real Hilbert space with norm 
$$
|F|_{\widehat{\F}}:=\sqrt{|F|_{\V}^2+|\hat\Lambda F|_{\V'}^2}.
$$ 
Let $\langle \cdot ,\cdot\rangle$ be the dualization between $\V'$ and $\V$. The 
{\it time-dependent generalized Dirichlet form} is now given through

\[ {\e}(F,G):= \left\{ \begin{array}{r@{\quad\quad}l}
 \A(F,G)- \langle \partial_t F,G\rangle & \mbox{ for}\ F \in{\F},\ G\in {\V} \\ 
  \A(F,G) - \langle \hat\Lambda G,F\rangle  & \mbox{ for}\ G\in{\widehat{\F}},\ F\in
            {\V.} \end{array} \right. \] \\ \\
Note that $\langle \cdot ,\cdot \rangle$ when restricted to $\H\times\V$ coincides with the inner product 
$(\cdot,\cdot)$ in $\H$.
In particular when $F\in C_0^{1}(E)$, $G\in \V$, then 
\begin{eqnarray}\label{form}
\e(F,G) & = & \frac{\sigma^2}{8} \int_{0}^{\infty}\int_{-\gamma(s)}^{\infty} 
\partial_x F(s,x) \,\partial_x G(s,x) \rho(s,x) dxds\nonumber  \\
& & \hspace*{+2cm}-\int_{0}^{\infty}\int_{-\gamma(s)}^{\infty}\partial_t F (s,x) G(s,x)\rho(s,x)dxds. 
\end{eqnarray}
For all corresponding objects to $\e$ which might not rigorously be defined here 
we refer to \cite{Tr4}. We also point out that the monotonicity assumption on 
$E_t$ as well as on the density $\rho$ in time is crucial for the construction of $\e$.\\ 
By \cite[\mbox{Lemma 1.6, Lemma 1.7}]{Tr4} the resolvent $(G_\alpha)_{\alpha>0}$ and the coresolvent  
$(\widehat G_\alpha)_{\alpha>0}$ associated with $\e$ are sub-Markovian and $C^1_0(E)\subset \F$ dense.
Let $\e_\alpha (F,G):=\e(F,G) + \alpha (F,G)$ for $\alpha >0$. Then 
$$
\e_\alpha (G_\alpha F,G)=(F,G)_{\H}=\e_\alpha (F,\widehat G_\alpha G)\ \ F,G\in \V.
$$
\begin{prop}\label{cons}
$(G_\alpha)_{\alpha>0}$ is Markovian, i.e. $G_{1} \mathbbm{I}_E = \mathbbm{I}_E$ 
$m$-a.e. 
\end{prop}
\begin{proof}
In order to prove the conservativity of $\e$ 
it is enough to show that for one 
$F\in \H\cap L^1(E,m)$, $F>0$ $m$-a.e., there exists 
$(W_n)_{n\ge 1}\subset {\cal{F}}$,  $0\le W_n\le \mathbbm{I}_E $, $n\ge 1$, $W_n\uparrow \mathbbm{I}_E $  
as $n\to\infty$, 
such that
$$
\lim_{n\to\infty}\e(W_n,\widehat{G}_1 F)=0.
$$
Indeed, if this is the case then
$$
0=\lim_{n\to\infty}\e(W_n,\widehat{G}_1 F)=\lim_{n\to\infty}\int_E (W_n-G_1 W_n)F \rho dxds=
\int_E (\mathbbm{I}_E-G_1 \mathbbm{I}_E)F \rho dxds,
$$ 
and $G_1 \mathbbm{I}_E=\mathbbm{I}_E$ as desired. We now fix $F$ as above, 
and determine below $(W_n)_{n\ge 1}$.\\
Let $g_n\in C_0^1(\mathbbm{R}^+)$, $u_n\in C_0^2(\mathbbm{R})$, $n\ge 1$, 
such that $0\le g_n,u_n\le 1$,  
$|\partial_t g_n|_{\infty}, |\partial_x u_n|_{\infty}\le L\cdot n^{-1}$, 
$|\partial_{xx} u_n|_{\infty}\le L\cdot n^{-2}$, where $L$ is some positive constant, and

\[ g_n(s)= \left\{ \begin{array}{r@{}l}
  & 1 \ \ \mbox{ if } \ s\in [0,n]\\ 
  & 0 \ \ \mbox{ if } \ s\in[2n,\infty),\\
\end{array} \right. \] \\ 
and 
\[ u_n(x)= \left\{ \begin{array}{r@{}l}
  & 1 \ \ \mbox{ if } \ [-\gamma(2n)]\le x\le [\lambda(0)+1]+n\\ 
  & 0 \ \ \mbox{ if } \ x\ge [\lambda(0)+1]+2n ,\\
\end{array} \right. \] \\
where $[x]:=\sup\{k\in \mathbbm{Z}|k\le x\}$. Then $W_n:=g_n u_n \mathbbm{I}_E\in \F$, $n\ge 1$, 
satisfies $W_n\uparrow \mathbbm{I}_E $ as $n\to\infty$, 
and since $\partial_x W_n(s,\beta(s))=\partial_x W_n(s,-\gamma(s))=0$  for all $s$, we easily find
$$
\e(W_n,\widehat{G}_1 F) = -\frac{\sigma^2}{8}\int_{0}^{\infty}\int_{-\gamma(s)}^{\infty}
\left ( \partial_{xx}W_n+
 \left (\frac{\delta-1}{(x+\gamma(s))}-bx \right)\partial_{x}W_n+
\frac{8}{\sigma^2}\partial_t W_n  \right ) \widehat{G}_1 F\rho dxds,
$$
so that $|\e(W_n,\widehat{G}_1 F)|$ is dominated by 
$$
\frac{L \cdot\sigma^2}{8}\int_{0}^{2n}\int_{[\lambda(0)+1]+n}^{[\lambda(0)+1]+2n}
\left ( \frac{1}{n^{2}}+\frac{\delta-1}{n([\lambda(0)+1]+n+\gamma(0))}+\frac{b([\lambda(0)+1]+2n)}{n} \right ) \widehat{G}_1 F\rho dxds
$$
$$
+\int_{0}^{2n}\int_{-\gamma(s)}^{\infty}\frac{L}{n}\,\widehat{G}_1 F\rho dxds.\\
$$
Noting that $\widehat{G}_1 F\rho dxds$ is a finite measure, we just apply Lebesgue's theorem, and 
the last sum is easily seen to converge to zero as $n\to \infty$. This concludes the proof.  
\end{proof}\\ \\
Let us define the strict capacity corresponding to $\e$. 
We fix $\Phi\in L^1(E,m)$, $0< \Phi \le 1$.  Let $(k G_1 \Phi\wedge 1)_U$ be the $1$-reduced function of 
$k G_1 \Phi\wedge 1:=min(k G_1 \Phi,1)$ on $U$, and let 
$$
\mbox{Cap}_{1,\widehat G_1\Phi}(U)=
\lim_{k\to \infty}\int_E (k G_1 \Phi\wedge 1)_U \Phi dm \ \mbox{ if } \ U\subset E \mbox{ is open}.
$$
If  $A\subset E$ arbitrary then 
$$
\mbox{Cap}_{1,\widehat G_1 \Phi}(A)=\inf\{\mbox{Cap}_{1,\widehat G_1\Phi}(U)|U\supset A, U \mbox{ open}\}.
$$ \\ 
We adjoin an extra point $\Delta$ to $E$ and let $E_{\Delta}:=E\cup\{\Delta\}$ be the one point compactification of $E$. As usual any 
function defined on $E$ is extended to $E_{\Delta}$ putting $f(\Delta)=0$. Given an increasing sequence $\fk$ of closed 
subsets of $E$, we define
$$
C_{\infty}(\{F_k\})=\{f:A\rightarrow \R\mid \bigcup_{k\ge 1}F_k\subset A\subset E,\, 
f_{\mid F_k\cup \{\Delta\}}\  is \ continuous \ \forall k\}.
$$ 
A subset $N\subset E$ is called strictly $\e$-exceptional if $\mbox{Cap}_{1,\widehat G_1 \Phi}(N)=0$. 
An increasing sequence $\fk$ of closed subsets of $E$ is called a strict $\e$-nest if 
$\mbox{Cap}_{1,\widehat G_1 \Phi}(F_k^c)\downarrow 0$ as $k\to\infty$.
A property of points in $E$ holds strictly $\e$-quasi-everywhere (s.$\e$-q.e.) 
if the property holds outside some strictly $\e$-exceptional set.
A function $f$ defined up to some strictly $\e$-exceptional set $N\subset E$ is called 
strictly $\e$-quasi-continuous (s.$\e$-q.c.)
if there 
exists a strict $\e$-nest $\fk$, such that $f\in C_{\infty}(\{F_k\})$. \\  
For a subset $A \subset E_\Delta$ let 
$\sigma_A := \inf\{t > 0 \mid \overline{Y}_t \in A\}$ (resp. $D_{A}=\inf\{t\ge 0|\overline{Y}_t\in A\}$) be the 
{\it first hitting time} 
(resp. {\it first entry time}) w.r.t. ${\mathbbm M}$. 
For a Borel measure $\nu$ on $E$ and a Borel set $B$ let $P_{\nu}(B):=\int_E P_y(B)\nu(dy)$ 
and $E_{\nu}$ be the expectation w.r.t. $P_{\nu}$.
As usual we denote by $E_y$ the expectation w.r.t. $P_y$. If $U\subset E$ is open, then 
\begin{eqnarray}\label{strictcap}
\mbox{Cap}_{1,\widehat G_1\Phi}(U)=\int_E E_{y}[e^{-\sigma_{U}}]\Phi(y) m(dy).
\end{eqnarray}
If $B\subset E$ is an arbitrary Borel measurable set, then
$$
\mbox{Cap}_{1,\widehat G_1\Phi}(B)=\int_E E_y[e^{-D_{B}}]\Phi(y) m(dy).
$$
Both follows from $\cite[\mbox{Lemma\, 0.8}]{Tr3}$.\\
By strict quasi-regularity every element in $\F$ admits a strictly $\e$-q.c. $m$-version 
(see \cite[\mbox{Proposition 0.9}]{Tr3}). 
For a subset $\D\subset \H$ denote by $\~ {\D}$  all the s.$\e$-q.c. $m$-versions of elements in $\D$. 
In particular $\~{\P}_{\F}$ denotes the set of all s.$\e$-q.c. $\rho dy$-versions of 1-excessive elements 
in $\V$ which are dominated by elements of $\F$. We have an analogy, 
namely \cite[\mbox{Theorem 0.16}]{Tr3}, to  
\cite[\mbox{Theorem\,2.3}]{Tr1}. That is: Let $\hat u\in\hatP_\hatF$. 
Then there exists a unique $\sigma$-finite
  and positive measure $\mu_{\hat u}$ on $(E,\B(E))$ 
charging no strictly
  $\e$-exceptional  set, such that
$$
\int_E \~ f \ d\mu_{\hat u} = \lim_{\alpha \to \infty} \e_1(f,\alpha
\widehat{G}_{\alpha+1}\hat u )\quad \forall \~ f\in\~ {\P}_{\F}-\~{\P}_{\F}\ .
$$
Also in  analogy to $\cite{Tr1}$ we introduce the following class of measures 
$$
\widehat{S}_{00}:=\{\mu_{{\hat u}}\mid\ \hat u \in \widehat \P_{\widehat  G_1\H_b^+}\ 
\mbox{and}\ \mu_{\hat{u}}(E)<\infty\} 
$$
where $\widehat  G_1\H_b^+:=\{\widehat  G_1 h\mid h\in \H_b^+\}$.\\ 
For $B\in \B(E)$ the following is known from \cite[\mbox{Theorem 0.17}]{Tr3}: 
$B$ is strictly $\e$-exceptional if, and only if $\mu(B)=0$  for all  $\mu$ in $\widehat{S}_{00}$.\\
Since $(\e, \F)$ is regular, i.e. $C_0(E)\cap\F$ is dense in $C_0(E)$ w.r.t. the uniform norm 
as well as in $\F$, 
it follows that 
$(\e, \F)$ is a (strictly) quasi-regular generalized Dirichlet form on 
$E$. On the other hand we can find a dense algebra of functions, namely  $C^1_0(E)$, in $\F$. 
These two facts imply the existence of a Hunt process associated to $\e$. 
Applying additionally Proposition \ref{cons}, and \cite[\mbox{Theorem 1.9}]{Tr4} we have:\\  
\begin{theo}\label{Hunt}
There exists a Hunt process ${\mathbbm M} = (\Omega, ({\cal F}_t)_{t\ge 0}, (\overline{Y}_t)_{t\ge
0},(P_{y})_{y=(s,x)\in E_\Delta})$ with state space E, and infinite life time, 
such that $R_{\alpha}F(s,x):=\int_0^{\infty}\int_{\Omega}e^{-\alpha t} F(\overline{Y}_t(\omega))P_{(s,x)}(d\omega)dt$ 
is a $\cal E$-q.c. 
$m$-version of $G_{\alpha}F$ for any $\alpha > 0$ and any $F \in {\cal H}_b$. 
Moreover there exists a $\e$-exceptional set $N\subset E$ such that
$$
P_{(s,x)}\left (t\mapsto \overline{Y}_t \mbox{ is continuous on } [0,\infty) \right )=1 
\mbox{ for every } (s,x)\in E \setminus N.
$$
\end{theo}\bigskip
We will see that ${\mathbbm M} = (\Omega, ({\cal F}_t)_{t\ge 0}, (\overline{Y}_t)_{t\ge
0},(P_{y})_{y=(s,x)\in E_\Delta})$ in fact solves (\ref{transform}).\bigskip
\subsubsection{Identification of the process $\overline{Y}$ associated to the Generalized Dirichlet form involving the Markovian local time}\label{GDF2}
\bigskip\medskip
We want to identify $\overline{Y}$. 
Let us first recall some basic definitions and facts about additive functionals related to generalized Dirichlet forms.\\ \\
A family $(A_t)_{t \ge 0}$ of extended real valued functions on $\Omega$ is called an
{\it additive functional} (abbreviated AF) of ${\mathbbm M} = (\Omega, ({\cal F}_t)_{t\ge 0}, (\overline{Y}_t)_{t\ge
0},(P_{y})_{y\in E_\Delta})$ (w.r.t. $\mbox{Cap}_{1,\widehat G_1\Phi}$), if:\\ \\
(i) $A_t(\cdot)$ is ${\cal F}_t$-measurable for all $t \ge 0$.\\ \\
(ii) There exists a {\it defining} set $\Lambda \in {\cal F}_{\infty}$ and 
a strictly \e-exceptional set $N \subset E$, such that 
$P_{y}(\Lambda)=1$ for all $y \in E \setminus N$, $\theta_t(\Lambda) \subset \Lambda$ 
for all $t > 0$ and for each $\omega \in \Lambda$, $t \mapsto A_t(\omega)$  
is right continuous on $[0, \infty)$ and has left  limits on $(0,\zeta(\omega))$, 
$A_0(\omega) = 0$, $|A_t(\omega)| < \infty$ for $t < \zeta(\omega)$, 
$A_t(\omega) = A_\zeta(\omega)$ for $t \ge \zeta(\omega)$ and 
$A_{t+s}(\omega)=A_t(\omega)+A_s(\theta_t \omega)$ for $s,t \ge 0$.\\ \\ 
An AF $A$ is called a {\it continuous additive functional} (abbreviated CAF), if
$t \mapsto A_t(\omega)$ is continuous on $[0,\infty)$, a 
{\it positive, continuous additive functional} (abbreviated PCAF) if 
$A_t(\omega) \ge 0$ and a finite AF, if $\mid A_t(\omega)\mid<\infty$ for all 
$t \ge 0, \omega \in \Lambda$. 
Two AF's $A$,$B$ are said 
to be equivalent (in notation $A=B$) if for each $t>0$ $P_{y}(A_t=B_t)=1$ for strictly $\e$-q.e. $y\in E$.
The {\it energy} of an AF $A$ of $\M$ is defined by
\begin{eqnarray}\label{energie}
e(A) & = & \lim_{\alpha\to\infty} \frac12\alpha^2 E_{\rho dy}\left [\int_{0}^{\infty}e^{-\alpha\,t}A_t^2dt\right ],
\end{eqnarray}
whenever this limit exists in $[0,\infty]$. We will set $\overline{e}(A)$ for the same expression but with 
$\overline{\lim}$ instead of $\lim$.\\ \\
Let $\widetilde{F}$ be a strictly $\e$-q.c. $\rho dy$-version of some element in $\H$.
The additive functional 
$$
A^{[F]}:=(\widetilde{F}(\overline{Y}_t)-\widetilde{F}(\overline{Y}_0))_{t\ge 0}
$$ 
is independent of the choice of $\widetilde{F}$ (i.e. defines the same equivalence class of AF's for 
any strictly $\e$-q.c. $\rho dy$-version $\widetilde{F}$ of $F$). 
The sub-Markovianity of $(\widehat G_\alpha)_{\alpha>0}$ implies
\begin{eqnarray*}
\overline{e} (A^{[F]}) & = & \overline{\lim_{\alpha\to \infty}}\left ( \alpha (F -\alpha G_\alpha F, F)_{\H}-
\frac{\alpha}{2}\int_E (F^2-\alpha G_{\alpha}F^2 )\rho dy\right )\\
& \le & \overline{\lim_{\alpha\to \infty}}\ \alpha (F -\alpha G_\alpha F, F)_{\H}.
\end{eqnarray*}
Since $\F\subset\V^{\F}$
(cf. e.g. proof of $\cite[\mbox{Lemma 3.1}]{Tr2}$) it follows 
$\lim_{\alpha\to \infty}\alpha  \widehat G_\alpha F=F$ weakly in $\V$. Hence 
$\lim_{\alpha\to \infty}\alpha (F -\alpha G_\alpha F, F)_{\H}=
\lim_{\alpha\to \infty}\e(F,\alpha  \widehat G_\alpha F)=\e(F,F)$ whenever $F\in \F$. 
In particular   
\begin{eqnarray}\label{999}
\overline{e} (A^{[F]}) & \le & 2|F|_{\F}^2 \ \mbox{ for any } F\in \F. 
\end{eqnarray}
Define
\begin{eqnarray*}
{\cal{M}}  & = & \{M|M \mbox{ is a finite AF}, E_{y}[M_t^2]< \infty, E_{y}[M_t]=0 \\
 &  &  \hspace*{+1.5cm} \mbox{ for strictly \e-q.e } y \in E\ \mbox{and all } t\ge 0\}.
\end{eqnarray*}
$M \in {\cal{M}}$ is called a {\it martingale additive functional} (MAF).
Furthermore define
\begin{eqnarray*}
\stackrel{\ \circ}{\cal{M}}  & = & \{M \in{\cal{M}} |\  e(M) < \infty\}.
\end{eqnarray*}
The elements of $\stackrel{\ \circ}{\cal{M}}$ are called MAF's of finite energy.\\  \\ 
Let $A$ be a PCAF of $\M$. Its Revuz measure $\mu_A$ (see $\cite[\mbox{Theorem 3.1}]{Tr1}$) is defined by 
\begin{eqnarray}\label{I3}  
\int_E G(y)\mu_A(dy) 
& = & \lim_{\alpha\to\infty} \alpha E_{\rho dy}\left [\int_{0}^{\infty}e^{-\alpha\,t}G(\overline{Y}_t)dA_t\right ] \mbox{ for all } G\in \B^{+}.
\end{eqnarray}
The dual predictable projection $\langle M \rangle$ of the square bracket of $M \in \stackrel{\ \circ}{\cal{M}}$ 
is a PCAF of $\M$. 
It then follows from ($\ref{energie}$), ($\ref{I3}$), that one half of the total mass of the Revuz measure $\mu_{\langle M\rangle}$ 
is equal to the energy of $M$, i.e.
\begin{eqnarray}\label{pcaf=quad}
e(M)=\frac{1}{2}\int_{E}\mu_{\langle M\rangle}(dy).
\end{eqnarray}
Therefore $\mu_{\langle M\rangle}$ is also called the energy measure of $M$. For $M$, $L\in\stackrel{\ \circ}{\cal{M}}$ let
$$
\langle M,L\rangle :=\frac{1}{2}\left (\langle M+L \rangle -\langle M\rangle-\langle L\rangle \right ).
$$
Then $(\langle M,L\rangle_t)_{t\ge 0}$ is a CAF of bounded variation on each finite interval.
Furthermore the finite signed measure $\mu_{\langle M,L\rangle}$ defined by 
$\mu_{\langle M,L\rangle}:=\frac{1}{2}(\mu_{\langle  M+L\rangle}-
\mu_{\langle M\rangle}-\mu_{\langle L\rangle})$ is related to $\langle M,L\rangle$ in the 
sense of ($\ref{I3}$). If $G \in\B_b^+$, then $\int_E G d\mu_{\langle \cdot,\cdot \rangle }$ is symmetric, 
bilinear and positive on $\stackrel{\ \circ}{\cal{M}}\times \stackrel{\ \circ}{\cal{M}}$.\\ \\
Define 
\begin{eqnarray*} 
{\cal{N}}_c & = & \{N| N \mbox{ is a finite CAF}, e(N) = 0, E_{y}[|N_t|]<\infty \\
&  & \hspace*{+1.5cm}  \mbox{ for strictly \e-q.e. } y \in E\ \mbox{and all } t\ge 0\}.
\end{eqnarray*}\\ 
For $F\in \F$, $A^{[F]}$ can uniquely 
be decomposed (see $\cite[\mbox{Theorem 4.5.(i)}]{Tr1}$, $\cite[\mbox{Remark 0.17}]{Tr3}$) as
\begin{eqnarray}\label{I5}
A^{[F]}
& = & M^{[F]}+N^{[F]}, \ \ M^{[F]}\in \stackrel{\ \circ}{\cal{M}}, \ \ N^{[F]}\in{\cal N}_c.
\end{eqnarray}
The identity (\ref{I5}) means that both sides are equivalent as additive functionals w.r.t. $\mbox{Cap}_{1,\widehat G_1\Phi}$. 
The uniqueness of (\ref{I5}) implies $a M^{[F]}+bM^{[G]}=M^{[aF+bG]}$, 
$a N^{[F]}+bN^{[G]}=N^{[aF+bG]}$, for any $a,b\in \R$, $F,G\in \F$.\\  
From Lemma 2.1 in \cite{Tr4} we know that for $F\in \F$ 
$$
\mu_{\langle M^{[F]}\rangle}(dxds)=\frac{\sigma^2}{4}(\partial_x F)^2 \rho dxds,
$$
and moreover, if $F$ is constant $\rho dy$-a.e. on a Borel set $B$. Then 
$$
\mu_{\langle M^{[F]}\rangle}(B)=0.
$$
Now let us come back to the identification of $\overline{Y}$. In order to identify the drift part we might proceed as follows. 
Denote by $\delta_x$ the Dirac measure in $x\in \R$. If $\beta(s)>-\gamma(s)$ a.e. $s$, 
then integrating by parts in 
(\ref{DF}) we obtain that the generator 
of the diffusion is given informally in the sense of distributions by 
\begin{eqnarray*}
LF(s,x) & = & 
\frac{\sigma^2}{8}\partial_{xx}F(s,x)+\frac{\sigma^2}{8}\frac{\delta-1}{(x+\gamma(s))}
\partial_x F(s,x)+ \partial_t F(s,x)+\nu_F(dx)ds\\ 
\end{eqnarray*}
where the boundary term $\nu_F$ is given by 
\begin{eqnarray}\label{smoo1}
\nu_F(dx)ds & = & \left \{p\partial_{x}^+ F(s,x)-(1-p)\partial_{x}^- F(s,x)\right \}\frac{\sigma^2}{8}|x+\gamma(s)|^{\delta-1}e^{-\frac{bx^2}{2}}\delta_{\beta(s)}(dx)ds \nonumber\\ 
&&-(1-p)\partial_{x}^+F(s,x)\frac{\sigma^2}{8}|x+\gamma(s)|^{\delta-1}e^{-\frac{bx^2}{2}} \delta_{-\gamma(s)}(dx)ds, \\ \nonumber
\end{eqnarray}
and where $\partial_{x}^+$, resp. $\partial_{x}^-$, denote the right hand, resp. the left hand 
derivative in space.\\ 
If we can show that $|x+\gamma(s)|^{\delta-1}e^{-\frac{bx^2}{2}}\delta_{\kappa(s)}(dx)ds$, $\kappa: \R^+ \rightarrow \R^+$ locally bounded 
and measurable, is a smooth measure w.r.t. $\cal{A}$, then there is a unique 
PCAF representing this measure 
by Theorem 2.2 in \cite{Tr4}. Theorem 2.3 in \cite{Tr4} then allows to identify the drift part. 
We will identify the corresponding diffusion when $\delta\ge 1$. \\
Let $\R^+\times \R=\bigcup_{n\ge 1}K_n$, where $(K_n)_{n\ge 1}$ be 
an increasing sequence of compact subsets of $\R^+\times \R$. Let $\overline{E}_n:=K_n \cap E$, $n\ge 1$.   
Since ${C^1_0(E)}\subset \F$ dense, it follows from $\cite[\mbox{III.Remark 2.11}]{St1}$ that 
$(\overline{E}_n)_{n\ge 1}$ is an $\e$-nest 
in the sense of $\cite[\mbox{III.Definition 2.3(i)}]{St1}$. Consequently, 
$P_y(\lim_{n\to \infty} \sigma_{\overline{E}_n^c}<\infty)=0$ for $\e$-q.e. $y\in E$, 
hence in particular for 
$\rho dy$-a.e. $y\in E$ (see $\cite[\mbox{IV. Lemma 3.10}]{St1}$).
We obtain that $(\overline{E}_n)_{n\ge 1}$ is an strict $\e$-nest by ($\ref{strictcap}$). 
$\cite[\mbox{Lemma 0.8(ii)}]{Tr3}$ now implies 
$P_y(\lim_{n\to \infty} \sigma_{\overline{E}_n^c}<\infty)=0$ for strictly $\e$-q.e. $y\in E$. 
We may without loss 
of generality assume that 
$\overline{E}_n\subset [0,n]\times \R\cap E$, $n \ge 1$, and that $\overline{E}_n$ is contained 
in the interior 
of $\overline{E}_{n+1}$ for any $n\ge 1$. 
From now on we will fix such a strict $\e$-nest $(\overline{E}_n)_{n\ge 1}$.\\
A proof following similar lines of arguments to the proof of the next proposition can be found in \cite{Tr4}. However, 
we include it for the readers convenience. 
\begin{prop}\label{smoo}
Let $\delta\ge 1$, $\kappa:\R^+ \to\R$ be measurable and locally bounded, 
such that $(s,\kappa(s))\in E_s$ for each $s\ge 0$. 
The measure 
$$
\mathbbm{I}_{\overline{E}_N}(s,x) |x+\gamma(s)|^{\delta-1}
e^{-\frac{bx^2}{2}}\delta_{\kappa(s)}(dx)ds, \ N\ge 1, 
$$
is smooth w.r.t. $(\A,\V)$.
\end{prop}
\begin{proof}
We only show the statement for $\delta>1$. The proof for $\delta=1$ works in the 
same manner and is even easier since the derivative of $y\mapsto |y|^{\delta-1}$ 
disappears as it is constant, 
so there are less additional terms (cf. below). 
Let $N\ge 1$. Let $F\in {C^1_0(E)}$, $\psi \in C_0^{\infty}(\R)$, $0\le \psi \le 1$, 
$|\partial_x \psi |_{\infty}\le 2$. Since  $\kappa$ is locally bounded, the following two 
values 
$$
\kappa_{Nmin}:=\inf\{\kappa(t)|t\in [0,N]\},\ \ \ \ \kappa_{Nmax}:=\sup\{\kappa(t)|t\in [0,N]\}, 
$$
are finite. Let $\psi=1$ on $[\kappa_{Nmin},\kappa_{Nmax}]$, 
$\psi=0$ on $[\kappa_{Nmax}+1,\infty[$. For $s\in [0,N] $ we have 
$$
F(s,\kappa(s))|\kappa(s)+\gamma(s)|^{\delta-1}e^{-\frac{b\kappa(s)^2}{2}}=
-\int_{\kappa(s)}^{\kappa_{Nmax}+1}\partial_x 
\left (\psi(x) F(s,x)|x+\gamma(s)|^{\delta-1}e^{-\frac{bx^2}{2}}\right )dx
$$
and thus
\begin{eqnarray}\label{987}
&& \hspace*{-1cm}\int_{\overline{E}_N}|F|(s,x)|x+\gamma(s)|^{\delta-1}e^{-\frac{bx^2}{2}}
\delta_{\kappa(s)}(dx)ds \nonumber  \\
&\le  &
\int_{0}^{N}\left |\int_{\kappa(s)}^{\kappa_{Nmax}+1}\partial_x 
\left (\psi(x)F(s,x)|x+\gamma(s)|^{\delta-1}e^{-\frac{bx^2}{2}}\right )dx\right |ds \nonumber \\
& \le & 2\int_{0}^{N} \int_{\kappa(s)}^{\kappa_{Nmax}+1}
\left (|\partial_x F|+|F|\right )|x+\gamma(s)|^{\delta-1}e^{-\frac{bx^2}{2}} dxds\nonumber \\ 
&&\hspace*{+2cm}+\int_{0}^{N} 
\int_{\kappa(s)}^{\kappa_{Nmax}+1} |F|\left ((\delta-1)|x+\gamma(s)|^{\delta-2}
-bx |x+\gamma(s)|^{\delta-1}\right )e^{-\frac{bx^2}{2}}dxds \nonumber \\
& \le &  C_N \sqrt{\A_{\frac{\sigma^2}{4}}(F,F)}+I(F),
\end{eqnarray}
with 
$$
I(F):=\int_{0}^{N} 
\int_{\kappa(s)}^{\kappa_{Nmax}+1} |F|\left ((\delta-1)|x+\gamma(s)|^{\delta-2}
-bx |x+\gamma(s)|^{\delta-1}\right )e^{-\frac{bx^2}{2}}dxds,
$$ 
and $C_N=\frac{8}{\sigma}\sqrt{\int_{0}^{N}\int_{\kappa(s)}^{\kappa_{Nmax}+1}
|x+\gamma(s)|^{\delta-1}e^{-\frac{bx^2}{2}}dxds}$. \\
Let $K\subset E$ be compact, and $\mbox{Cap}^{\A}(K)=0$. By ($\ref{capa}$) 
\begin{eqnarray*}
\mbox{Cap}^{\A}(K)=\inf\{\A_1(F,F);F\in C^{1}_{0,K}(E)\}, 
\end{eqnarray*}
where $C^{1}_{0,K}(E)=\{F\in {C^1_0(E)}| F(s,x)\ge 1, \forall (s,x)\in K\}$. Hence, there 
exists $(F_n)_{n\in\N}\subset {C^1_0(E)}$, $F_n(s,x)\ge 1$, for every 
$n\in\N$, $(s,x)\in K$, such that $|F_n|_{\V}\to 0$ as $n\to\infty$. 
Since normal contractions operate on $\V$ we may assume that 
$\sup_{n\in \N}\sup_{(s,x)\in K}|F_n(s,x)|\le C$. 
Selecting a subsequence if necessary we may also assume that 
$\lim_{n\to\infty}|F_n|=0$ $\rho(s,x) dxds$-a.e, hence  $dxds$-a.e. Consequently, 
using Lebesgue's theorem we obtain
$$
I(F_n)\to 0 \mbox{ as } n\to\infty .
$$  
Therefore by ($\ref{987}$)
\begin{eqnarray*}
&&\int_{\overline{E}_N}1_K(s,x)|x+\gamma(s)|^{\delta-1}e^{-\frac{bx^2}{2}}
\delta_{\kappa(s)}(dx)ds \\
&\le& \limsup_{n\to\infty}
\int_{\overline{E}_N}|F_n|(s,x)|x+\gamma(s)|^{\delta-1}e^{-\frac{bx^2}{2}}\delta_{\kappa(s)}(dx)ds\\
& \le & \limsup_{n\to\infty}\left \{ C_N \sqrt{\A_{\frac{\sigma^2}{4}}(F_n,F_n)}+I(F_n)\right \}=0
\end{eqnarray*}
Since $1_{\overline{E}_N}(s,x) |x+\gamma(s)|^{\delta-1}e^{-\frac{bx^2}{2}}\delta_{\kappa(s)}(dx)ds$, 
as well as $\mbox{Cap}^{\A}$ are inner regular we obtain that 
the measure $1_{\overline{E}_N}(s,x) |x+\gamma(s)|^{\delta-1}e^{-\frac{bx^2}{2}}\delta_{\kappa(s)}(dx)ds$ 
is smooth w.r.t. $(\A,\V)$.
\end{proof}\\ \\ \\
Let us choose  $(J_M)_{M\ge 1}$, 
$(H_{M})_{M\ge 1}\subset {C^2_0(E)}$, with 

\[ H_M(s,x):= \left\{ \begin{array}{r@{\quad}l}
 x&\mbox{for}\ (s,x)\in \overline{E}_M \\ 
 0&\mbox{for}\ (s,x)\in \overline{E}_{M+1}^c,  \end{array} \right. \] \\ 
$M\ge 1$, and  

\[ J_M(s,x):= \left\{ \begin{array}{r@{\quad}l}
 s&\mbox{for}\ (s,x)\in \overline{E}_M \\ 
 0&\mbox{for}\ (s,x)\in \overline{E}_{M+1}^c,  \end{array} \right. \] \\ 
$M\ge 1$. Let further 
$$
H(s,x):=x,
$$
and 
$$
J(s,x):=s.
$$\\
$(H_{M})_{M\ge 1}$ (resp. $(J_{M})_{M\ge 1}$) is a localizing sequence for $H$ (resp. $J$). Obviously 
$$
A_{t\wedge \sigma_{\overline{E}_M^c}}^{[H_K]}=A_{t\wedge \sigma_{\overline{E}_M^c}}^{[H_L]} \mbox{for any } K\ge L\ge M.
$$
We claim that 
$$
M_{t\wedge \sigma_{\overline{E}_M^c}}^{[H_K]}=M_{t\wedge \sigma_{\overline{E}_M^c}}^{[H_L]} \mbox{for any } K\ge L\ge M.
$$\\
Indeed, for strictly $\e$-q.e $y\in E$, and any $t\ge 0$,\\
\begin{eqnarray*}
E_{y}\left [\langle M^{[H_K-H_L]}\rangle_{t\wedge\sigma_{\overline{E}_M^c}}\right ]
& = & E_{y}\left [\int_{0}^{t\wedge\sigma_{\overline{E}_M^c}} 1_{\overline{E}_M}(\overline{Y}_s)\,d\langle M^{[H_K-H_L]} \rangle_{s}\right ]\\
& \le & E_{y}\left [\int_{0}^{t}1_{\overline{E}_M}(\overline{Y}_s)\,d\langle M^{[H_K-H_L]} \rangle_{s}\right ].\\
\end{eqnarray*}
By Lemma 2.1(ii) in \cite{Tr4} $\mu_{\int_{0}^{\cdot}1_{\overline{E}_M}(\overline{Y}_s)d\langle M^{[H_K-H_L]} \rangle_s}=
\mu_{\langle M^{[H_K-H_L]} \rangle}(\overline{E}_M)=0$.
Thus by injectivity of the Revuz-correspondence (see \cite[\mbox{Remark 5.2(ii)}]{Tr1}) 
$E_{y}\left [ {\int_{0}^{t}1_{\overline{E}_M}(\overline{Y}_s)\,d\langle M^{[H_K-H_L]} \rangle_{s}}\right ]=0$ strictly 
$\e$-q.e. $y\in E$. Hence the same 
is true for $E_{y}\left [\langle M ^{[H_K-H_L]} \rangle_{t\wedge\sigma_{\overline{E}_M^c}}\right ] $. 
We know that $\left ((M^{[H_K-H_L]}_t)^2-\langle M ^{[H_K-H_L]} \rangle_{t}\right )_{t\ge 0}$ is
 a martingale w.r.t. $P_y$ for strictly $\e$-q.e. $y\in E$. The optional sampling theorem then implies 
$$
E_{y}\left [(M^{[H_K-H_L]}_{t\wedge\sigma_{\overline{E}_M^c}})^2\right ] =E_{y}\left [\langle M^{[H_K-H_L]} 
\rangle_{t\wedge\sigma_{\overline{E}_M^c}}\right ] =0
$$   
for strictly $\e$-q.e. $y\in E$ and the claim is shown. The analogous statements hold for $A^{[J_K]}$, $M^{[J_K]}$. Thus we may set 
$$
M^{[H]}_t:=\lim_{M\to \infty}M^{[H_M]}_t\hspace*{2cm}N^{[H]}_t:=A^{[H]}_t-M^{[H]}_t,
$$
and 
$$
M^{[J]}_t:=\lim_{M\to \infty}M^{[J_M]}_t\hspace*{2cm}N^{[J]}_t:=A^{[J]}_t-M^{[J]}_t,
$$
in order to obtain
$$
A^{[H]}_t=M^{[H]}_t+N^{[H]}_t,\hspace*{+2cm} A^{[J]}_t=M^{[J]}_t+N^{[J]}_t.
$$
Note that $N^{[H]}_t=\lim_{M\to \infty}N^{[H_M]}_t$, $N^{[J]}_t=\lim_{M\to \infty}M^{[J_M]}_t$. We want to find the 
explicit expressions for  $M^{[J]}$, $N^{[J]}$, $M^{[H]}$, $N^{[H]}$. 
Let $F\in C^2_0(E)$. Integrating by parts we obtain for any $G\in C^1_0(E)$
\begin{eqnarray}\label{id1}
-\e(F,G) & = & \int_{E} \left ( \frac{\sigma^2}{8}\left (\partial_{xx}F+
\left (\frac{\delta-1}{(x+\gamma(s))}-bx\right )\partial_x F\right )
+\partial_t F\right)G\rho dxds\nonumber  \\ 
 &&-(1-p)\frac{\sigma^2}{8} \int_{0}^{\infty}\int_{-\gamma(s)}^{\beta(s)}
\partial_x\left (\partial_x F\,G |x+\gamma(s)|^{\delta-1}e^{-\frac{bx^2}{2}}\right ) dxds \nonumber \\  
&&-p\frac{\sigma^2}{8}\int_{0}^{\infty}\int_{\beta(s)}^{\infty}
\partial_x\left (\partial_x F\,G |x+\gamma(s)|^{\delta-1}e^{-\frac{bx^2}{2}}\right )  dxds \nonumber \\ 
&=&\int_{E} \left ( \frac{\sigma^2}{8}\left (\partial_{xx}F
+\left (\frac{\delta-1}{(H+\gamma\circ J)}-bH\right )\partial_x F\right )
+\partial_t F\right)G\rho dxds \nonumber \\ 
&&\hspace*{-2cm}+ \frac{\sigma^2}{8}\int_{E} G\partial_x F 
\left \{(1-p)\,\mathbbm{I}_{\{\beta(s)>-\gamma(s)\}}  
+p\mathbbm{I}_{\{\beta(s)=-\gamma(s)\}}\right \} |x+\gamma(s)|^{\delta-1}e^{-\frac{bx^2}{2}}
\delta_{-\gamma(s)}(dx)ds 
\nonumber \\ 
&&+\frac{\sigma^2}{8}\int_{E} G\partial_x F\,(2p-1) \mathbbm{I}_{\{\beta(s)>-\gamma(s)\}}
|x+\gamma(s)|^{\delta-1}e^{-\frac{bx^2}{2}}\delta_{\beta(s)}(dx)ds.
\end{eqnarray}\\   
Obviously, ($\ref{id1}$) extends to $G\in \V_b$. By Proposition $\ref{smoo}$, the measure 
$$
\mathbbm{I}_{\overline{E}_M}(s,x)\frac{\sigma^2}{8}|x+\gamma(s)|^{\delta-1}e^{-\frac{bx^2}{2}}
\delta_{\kappa(s)}(dx)ds,\ M\ge 1, \ \delta\ge 1,
$$ 
$\kappa=-\gamma, \beta$, is smooth w.r.t. $(\A,\V)$.
Let $\ell_t^{\kappa}$ denote the unique positive continuous additive functional (PCAF) of 
$\overline{Y}$ associated to $\frac{\sigma^2}{8}|x+\gamma(s)|^{\delta-1}e^{-\frac{bx^2}{2}}\delta_{\kappa(s)}(dx)ds$ 
(see Theorem 2.2 in \cite{Tr4}). 
Then $\int_0^t G(\overline{Y}_s)d\ell_s^{\kappa}$ is associated to 
$G(s,x)\frac{\sigma^2}{8}|x+\gamma(s)|^{\delta-1}e^{-\frac{bx^2}{2}}\delta_{\kappa(s)}(dx)ds$ 
for any $G\in \B_b(E)$. In particular,  $\ell_t^{-\gamma}$ vanishes, if $\delta\not=1$.
We obtain 
\begin{eqnarray}\label{id2}
N^{[F]}_t&=&\int_0^t \left (\frac{\sigma^2}{8}\left (\partial_{xx}F+
\left (\frac{\delta-1}{H+\gamma\circ J}-bH\right )\partial_x F\right )+
\partial_t F \right )(\overline{Y}_s)ds \nonumber \\ 
&&+(2p-1)\int_0^t \partial_x F\, \mathbbm{I}_{\{\beta\circ J>-\gamma\circ J\}}(\overline{Y}_s)d\ell_s^{\beta}
\nonumber \\ 
&&+\mathbbm{I}_{\{\delta= 1\}}\int_0^t \partial_x F
\left \{ (1-p)\mathbbm{I}_{\{\beta\circ J>-\gamma\circ J\}} 
+p\mathbbm{I}_{\{\beta\circ J=-\gamma\circ J\}}\right \}(\overline{Y}_s)d\ell_s^{-\gamma}.
\end{eqnarray}\\
Indeed, if we denote the r.h.s. of ($\ref{id2}$) by $A_t$ then in particular by ($\ref{id1}$)
$-\e(F,\widehat G_1 W)=\lim_{\alpha\to\infty} \alpha^2 E_{\widehat G_1 W \rho dy}
\left [\int_{0}^{\infty}e^{-\alpha\,t}A_tdt\right ]$ for all $W\in \H_b$. Hence 
$N^{[F]}_t=A_t$ by Theorem 2.3 in \cite{Tr4}. On the other hand, 
by Lemma 2.1(i) in \cite{Tr4} the 
Revuz measure $\mu_{\langle M^{[F]}\rangle}$ 
is equal to $\frac{\sigma^2}{4}(\partial_x F)^2\rho dy$. A simple calculation shows that the Revuz 
measure of $\frac{\sigma^2}{4}\int_0^{t} (\partial_x F)^2(\overline{Y}_s)ds$ 
is also equal to $\frac{\sigma^2}{4}(\partial_x F)^2\rho dy$. 
Consequently, we have $\langle M^{[F]}\rangle_t=\frac{\sigma^2}{4}\int_0^{t} (\partial_x F)^2(\overline{Y}_s)ds$ (see 
\cite[\mbox{Remark 5.2(ii)}]{Tr1}) and therefore we may assume that
\begin{eqnarray}\label{id3}
M^{[F]}_t & = & \frac{\sigma}{2}\int_0^t \partial_x F(\overline{Y}_s)dB_s
\end{eqnarray}
with $((B_t)_{t\ge 0}, P_y, (\F_t)_{t\ge 0})$ being a Brownian motion starting at zero for 
strictly $\e$-q.e. $y\in E$.\\
By ($\ref{id1}$), ($\ref{id2}$), applied to $J_M$, letting $M\to \infty$, we obtain
$$
J(\overline{Y}_t)=J(\overline{Y}_0)+t.
$$
We put 
$$
X_t:=H(\overline{Y}_t), \ t\ge 0,
$$ 
so that
$$
\overline{Y}_t=(J(\overline{Y}_0)+t, X_t).
$$\\
Applying again ($\ref{id1}$), ($\ref{id2}$), but this time  
to $H_M$, letting $M\to \infty$, we obtain 
\begin{eqnarray}\label{id4}
X_t &=&  X_0+\frac{\sigma}{2}B_t+\frac{\sigma^2}{8}\int_0^t\frac{\delta-1}{X_s+\gamma(J(\overline{Y}_s))}-bX_s ds
+(2p-1)\int_0^t\mathbbm{I}_{\{\beta\circ J>-\gamma\circ J\}}(\overline{Y}_s)d\ell_s^{\beta} \nonumber \\
&&+\mathbbm{I}_{\{\delta= 1\}}\int_0^t
\left \{(1-p)\mathbbm{I}_{\{\beta\circ J>-\gamma\circ J\}}
+p\mathbbm{I}_{\{\beta\circ J=-\gamma\circ J\}}\right \}(\overline{Y}_s)d\ell_s^{-\gamma}.\\ \nonumber 
\end{eqnarray}
(\ref{id4}) holds $P_y$-a.s for $\e$-q.e. $y\in E$.\bigskip
\subsubsection{Identification of the process $\overline{Y}$ associated to the Generalized Dirichlet form involving the semimartingale local time}\label{GDF3}
\bigskip\medskip
The positive increasing processes $\ell^{\beta}, \ell^{-\gamma}$ are so-called Markovian local times because they are PCAF's of a strong Markov process (namely $\overline{Y}$) with 
related $1$-potentials and associated smooth measures that are singular (cf. (\ref{smoo1}), Proposition \ref{smoo}, and the Revuz formula (\ref{I3})).
It is a priori not clear whether these Markovian local times are semimartingale local times, i.e. increasing processes that occur in Tanaka's formula. There are Markov processes that are 
semimartigales and where the Markovian local times and the semimartingale local times are different to each other (see e.g. \cite{ay}) but this will not be the case in our situation (see Remark \ref{constloc}(ii)). 
\\ \\
\begin{rem}\label{constloc}
(i) If $\delta>1$, then 
$(x+\gamma(s))^{-1}\rho(s,x)dxds$ is a smooth measure,  
since $(x+\gamma(s))^{-1}\in L^1_{loc}(E,\rho dy)$, and therefore  $X$ is a semimartingale. 
If $\delta=1$, $\int_0^t\frac{\delta-1}{X_s+\gamma(J(\overline{Y}_s))}ds$ disappears, and $X$ is 
again a semimartingale.\\
If we had admitted $\delta<1$, then clearly $(x+\gamma(s))^{-1}\notin L^1_{loc}(E,\rho dy)$, 
hence $X$ would not be a semimartingale, and we would have to work with principal 
values in (\ref{id4}). Further Proposition \ref{smoo} wouldn't apply. Nonetheless, it is clearly 
possible to consider the case $\delta<1$ (see Remark \ref{extension} below).\\
Finally, since all subsequent transformations applied to $X$ ($\delta\ge 1$!) are keeping 
the class of semimartingales invariant, the processes $Y,Z,$ constructed below, and renamed as 
$R,\sqrt{R}$, in the introduction, will remain semimartingales.\\ 
(ii) In our case it will turn out that if $\kappa$ is regular enough, e.g. $\kappa\in H^{1,2}_{loc}(\R^+)$, 
then $\ell_t^{\kappa}$ (restricted to its support) is a constant multiple of a 
classical semimartingale local time (see e.g. \cite[\mbox{chapter VI}]{RYor} for a rigorous definition of semimartingale local time).
We will determine these constants for $\kappa=\beta$, $\kappa=-\gamma$, but 
we remark that everything would have worked exactly in the same way up to (\ref{id4}) 
if we only assumed $\beta$, $-\gamma$, to be measurable and decreasing. The later Girsanov transformation of section \ref{GIR}
for which only $\gamma\in  H^{1,2}_{loc}(\R^+)$ is needed, can also in this case be applied to (\ref{id4})  (see Remark \ref{addcalc}(ii)).\\ \\ 
\end{rem}
In order to determine the constants mentioned in Remark \ref{constloc}(ii), 
we will evaluate the l.h.s. of Tanaka's formula (\ref{tanaka}) below with the help of (\ref{I5}) and then compare the result with the r.h.s. of (\ref{tanaka}) (see (\ref{fuku}), (\ref{loc2})). 
We consider the point-symmetric derivative 
\[sgn(x):= \left\{ \begin{array}{r@{\quad}l}
 1&\mbox{if }\ x>0 \\ 
 0&\mbox{if }\ x=0\\ 
 -1&\mbox{if }\ x<0, \end{array} \right. \] \\ 
the left continuous derivative 
\[\overline{sgn}(x):= \left\{ \begin{array}{r@{\quad}l}
 1&\mbox{if }\ x>0 \\  
 -1&\mbox{if }\ x\le 0,   \end{array} \right.\\ \] \\ 
and the right continuous derivative
\[\underline{sgn}(x):= \left\{ \begin{array}{r@{\quad}l}
 1&\mbox{if }\ x\ge0 \\  
 -1&\mbox{if }\ x< 0,   \end{array} \right.\\ \] \\
of $|x|$. Let $\kappa$ be continuous, and locally of bounded variation. Since $X-\kappa$ is 
a continuous semimartingale, we may apply Tanaka's formula 
\begin{eqnarray}\label{tanaka}
|X_t-\kappa(t)|=|X_0-\kappa(0)|+\int_0^t f(X_s-\kappa(s))dX_s+\ell_t^{0f}(X-\kappa),
\end{eqnarray}
(cf. e.g. \cite[\mbox{VI.(1.2) Theorem, (1.25) Exercise}]{RYor}), where
\[\ell_t^{0f}(X-\kappa)= \left\{ \begin{array}{r@{\quad}l}
 \ \ell_t^{0}(X-\kappa)  &\mbox{if }\ f=sgn \\ 
 \ell_t^{0+}(X-\kappa) &\mbox{if }\ f=\overline{sgn}\\ 
 \ell_t^{0-}(X-\kappa) &\mbox{if }\ f=\underline{sgn}, \end{array} \right. \] \\ 
and $\ell_t^{0}(X-\kappa)$ (resp. $\ell_t^{0+}(X-\kappa)$, $\ell_t^{0-}(X-\kappa))$, is called the
{\it symmetric local time} (resp. {\it upper local time}, {\it lower local time}) in zero of the continuous 
semimartingale $X-\kappa$. In the book \cite{RYor} they authors decided to work with the left 
continuous derivative of $x\mapsto |x|$, and they use the expression $L_t^0$ for our $\ell_t^{0+}$.
For our framework it is more intuitive to work with the point symmetric derivative 
(see Remark \ref{whysym} below). \\
By \cite[\mbox{VI.(1.9) Corollary, (1.25) Exercise}]{RYor} 
we have $P_y$-a.s. for $\e$-q.e. $y\in E$
$$
\ell_t^{0+}(X-\kappa)=\lim_{\varepsilon\downarrow 0}
\frac{1}{\varepsilon}\int_0^t\mathbbm{I}_{[0,\varepsilon)}(X_s-\kappa(s))d\langle X,X\rangle_s,
$$
$$
\ell_t^{0-}(X-\kappa)=\lim_{\varepsilon\downarrow 0}
\frac{1}{\varepsilon}\int_0^t\mathbbm{I}_{(\varepsilon,0]}(X_s-\kappa(s))d\langle X,X\rangle_s,
$$
and
$$
\ell_t^{0}(X-\kappa)=\frac{\ell_t^{0+}(X-\kappa)+\ell_t^{0-}(X-\kappa)}{2}.
$$\\
From \cite{Tr6} we know that $\{s\ge 0|X_s-\kappa(s)=0\}$ is 
$P_y$-a.s. of Lebesgue measure zero for $\e$-q.e. $y\in E$.\\ \\ 
Let $(F_M)_{M\ge 1}\subset {C^2_0(E)}$, such that 
\[ F_M(s,x):= \left\{ \begin{array}{r@{\quad}l}
 1&\mbox{for}\ (s,x)\in \overline{E}_M \\ 
 0&\mbox{for}\ (s,x)\in \overline{E}_{M+1}^c,  \end{array} \right. \] 
$M\ge 1$, and let 
$$
\overline{\kappa}(s,x):=|x-\kappa(s)|, \ \ \kappa\in H^{1,2}_{loc}(\R^+).\\
$$ \\
It is easy to see that $\overline{\kappa} F_M\in \F$ 
for any $M\ge 1$. We will use the same 
localization procedure 
as before. Thus, if  $M_t^{[\overline{\kappa}]}:=\lim_{M\to \infty} M_t^{[\overline{\kappa} F_M]}$, 
and $N_t^{[\overline{\kappa}]}:=
\lim_{M\to \infty} N_t^{[\overline{\kappa} F_M]}$, then $A_t^{[\overline{\kappa}]}=
M_t^{[\overline{\kappa}]}+N_t^{[\overline{\kappa}]}$. 
If $G\in C^1_0(E)$, then
$$
-\e(\overline{\beta} F_M,G)= \frac{\sigma^2}{8}\int_{E} \overline{\beta} \partial_{xx}F_M G\rho dxds
+ \frac{\sigma^2}{8}\int_{E} 2sgn(\overline{\beta})\partial_{x}F_M G\rho dxds \nonumber \\
$$
$$
+ \frac{\sigma^2}{8}\int_{E} \left (\frac{\delta -1}{\overline{\gamma}}-bH \right )
\left (\overline{\beta}\partial_{x}F_M+sgn(\overline{\beta})F_M\right ) G\rho dxds 
$$
$$
-(1-p)\int_E (\overline{\beta}\partial_{x}F_M-F_M) G \mathbbm{I}_{\{\beta(s)>-\gamma(s)\}}
\frac{\sigma^2}{8}|x+\gamma(s)|^{\delta -1}e^{-\frac{bx^2}{2}}\delta_{\beta(s)}(dx)ds 
$$
$$
+p\int_E  (\overline{\beta}\partial_{x}F_M+F_M)G \frac{\sigma^2}{8}|x+\gamma(s)|^{\delta -1}e^{-\frac{bx^2}{2}}\delta_{\beta(s)}(dx)ds 
$$
$$
-(1-p)\int_E (\overline{\beta}\partial_{x}F_M-F_M)sgn(\overline{\beta}) G 
\mathbbm{I}_{\{\beta(s)>-\gamma(s)\}}
\frac{\sigma^2}{8}|x+\gamma(s)|^{\delta -1}e^{-\frac{bx^2}{2}}\delta_{-\gamma(s)}(dx)ds 
$$
$$
+\int_E \left ( |x-\beta(s)|\partial_t F_M -F_M sgn\left ( x-\beta(s)\right ) \beta'(s) \right ) G\rho dy.
$$
Obviously, the last equation extends to $G\in \V_b$. Thus, letting $M\to\infty$, 
\begin{eqnarray}
N_t^{[\overline{\beta}]}&= &\frac{\sigma^2}{8}\int_{0}^{t} 
\left (\frac{\delta -1}{H+\gamma\circ J}-bH \right )sgn(\overline{\beta})(\overline{Y}_s)ds-
\int_{0}^{t}sgn(\overline{\beta})(\overline{Y}_s)d\beta(J(\overline{Y}_s)) \nonumber\\
&&+(1-p)\int_0^t\mathbbm{I}_{\{\beta\circ J>-\gamma\circ J\}}(\overline{Y}_s)d\ell_s^{\beta}
+p\ell_t^{\beta}\nonumber\\\
&&+(1-p)\int_0^t sgn(\overline{\beta})(\overline{Y}_s)\mathbbm{I}_{\{\beta\circ J>-\gamma\circ J\}}(\overline{Y}_s)d\ell_s^{-\gamma}.
\end{eqnarray} 
On the other hand by Lemma 2.1(i) in \cite{Tr4}
\begin{eqnarray*}
&&\hspace*{+2cm}\mu_{\langle M^{[\overline{\beta}F_M]}\rangle} = 
\frac{\sigma^2}{4}\partial_x\left ( |x-\beta(s)| F_M \right )^2\rho dy\\
&=&\frac{\sigma^2}{4}\left (\left ( |x-\beta(s)|\partial_x F_M \right )^2+
2|x-\beta(s)|\partial_x F_M sgn\left ( x-\beta(s) \right ) F_M \right )\rho dy\\
&& \hspace*{+2cm}+ \frac{\sigma^2}{4}\left (F_M sgn\left ( x-\beta(s) \right )\right )^2\rho dy.
\end{eqnarray*} 
We obtain $\langle M^{[\overline{\beta}]}\rangle_t=\frac{\sigma^2}{4}\int_{0}^{t}sgn(\overline{\beta})(\overline{Y}_s)^2ds$. 
Consequently, we may assume that
$$
M^{[\overline{\beta}]}_t=\frac{\sigma}{2}\int_{0}^{t}sgn(\overline{\beta})(\overline{Y}_s)dB_s.
$$ 
Note that $\int_{0}^{t}sgn(\overline{\beta})(\overline{Y}_s)d\ell_s^{\beta}=0$, because for its associated signed 
smooth measure we have $sgn(x-\beta(s))\delta_{\beta(s)}(dx)ds=0$, since $sgn(0)=0$. Therefore
\begin{eqnarray}\label{fuku}
|X_t-\beta(J(\overline{Y}_t))| & = & |X_0-\beta(J(\overline{Y}_0))|+
\int_{0}^{t}sgn(X_s-\beta(J(\overline{Y}_s)))d(X_s-\beta(J(\overline{Y}_s)))\nonumber \\
&&+p\ell_t^{\beta}+
(1-p)\int_0^t\mathbbm{I}_{\{\beta\circ J>-\gamma\circ J\}}(\overline{Y}_s)d\ell_s^{\beta}.\\ \nonumber
\end{eqnarray} 
For $\kappa:\R^+\to\R$, and $u\in \R^+$, define 
$$
\kappa_u(t):=\kappa(u+t).
$$
Recall that 
$$
P_y(J(\overline{Y}_t)=J(y)+t)=1,
$$
so that 
$$
\kappa_{J(y)}(t)=\kappa(J(\overline{Y}_t))\ \ P_y\mbox{-a.s.} 
$$\\
Comparing (\ref{fuku}) with Tanaka's formula (\ref{tanaka}) we see that 
\begin{eqnarray}\label{loc2}
\ell_t^{0}(X-\beta_{J(y)})=p\ell_t^{\beta}+
(1-p)\int_0^t\mathbbm{I}_{\{\beta_{J(y)}(s)>-\gamma_{J(y)}(s)\}}d\ell_s^{\beta}
\hspace*{+1cm}P_y\mbox{-a.s}
\end{eqnarray} 
for $\e$-q.e. $y\in E$.
In a similar way one can see that we have 
\begin{eqnarray}\label{loc3}
\ell_t^{0}(X+\gamma_{J(y)})
& = & 
\frac12\ell_t^{0+}(X+\gamma_{J(y)})\nonumber \\
& = & \int_0^t p\mathbbm{I}_{\{\beta_{J(y)}(s)=-\gamma_{J(y)}(s)\}}+
(1-p)\mathbbm{I}_{\{\beta_{J(y)}(s)>-\gamma_{J(y)}(s)\}}d\ell_s^{-\gamma},
\end{eqnarray} 
$P_y$-a.s. for $\e$-q.e. $y\in E$. Therefore, (\ref{id4}) rewrites $P_y$-a.s. as 
\begin{eqnarray}\label{rootCIR} 
X_t & = & X_0+\frac{\sigma}{2}B_t+\frac{\sigma^2}{8}\int_0^t\frac{\delta-1}{X_s+\gamma_{J(y)}(s)}-bX_s ds
\nonumber\\ 
&&+(2p-1)\int_0^t\mathbbm{I}_{\{\beta_{J(y)}(s)>-\gamma_{J(y)}(s)\}}d\ell_s^0(X-\beta_{J(y)}) 
+\frac{\mathbbm{I}_{\{\delta= 1\}}}{2}\ell_t^{0+}(X+\gamma_{J(y)}).\\ \nonumber 
\end{eqnarray} 
\subsection{Girsanov transformation of (\ref{transform}) and concluding remarks}\label{GIR}
\bigskip\medskip
Let $b\in \R^+$. We define 
$$
W_t:=B_t+\frac{1}{4\sigma }\int_0^t 8\gamma'_{J(\overline{Y}_0)}(s)+\sigma^2 b\gamma_{J(\overline{Y}_0)}(s)ds,
$$ 
and 
$$
dQ_y=e^{-\frac{1}{4\sigma }\int_0^t 8\gamma'_{J(y)}(s)+\sigma^2 b\gamma_{J(y)}(s)dB_s-
\frac{1}{32\sigma^2}\int_0^t |8\gamma'_{J(y)}(s)+\sigma^2 b\gamma_{J(y)}(s)|^2 ds} dP_y 
\mbox{ on } \F_t.
$$ 
Obviously, Novikov's condition is satisfied since $\gamma\in H^{1,2}_{loc}(\R^+)$, hence 
$B_t$ is a Brownian motion under the equivalent measure $Q_y$. 
Put 
$$
Y_t:=X_t+\gamma_{J(Y_0)}(t).  
$$
Then
$$
\ell_t^0(X-\beta_{J(y)})=\ell_t^0(X+\gamma_{J(y)}-(\beta_{J(y)}+\gamma_{J(y)}))=
\ell_t^0(Y-\lambda_{J(y)})\ \ \ \ Q_y\mbox{-a.s},
$$ 
since $Q_y$ is equivalent to $P_y$. Analogously, $\ell_t^{0+}(X+\gamma_{J(y)})=\ell_t^{0+}(Y)$ holds 
$Q_y$-a.s.  
Thus, under $Q_y$, $Y_t\ge 0$ (since $X_t\ge -\gamma(t)$), and
\begin{eqnarray}\label{CIRroot} 
Y_t & = & Y_0+\frac{\sigma}{2}W_t
+\frac{\sigma^2}{8}\int_0^t\frac{\delta-1}{Y_s}-bY_s ds\nonumber \\
&+& (2p-1)\int_0^t\mathbbm{I}_{\{\lambda_{J(y)}(s)>0\}}d\ell_s^0(Y-\lambda_{J(y)})
\ + \ \frac{\mathbbm{I}_{\{\delta= 1\}}}{2}\ell_t^{0+}(Y), 
\end{eqnarray}
where $Q_y(Y_0=H(y)+\gamma(J(y)))=1$.
Moreover, under $Q_y$, $Z_t:=Y_t^2$, satisfies
\begin{eqnarray}\label{CIR} 
Z_t &=& Z_0+\sigma\int_0^t\sqrt{Z_s}dW_s+\frac{\sigma^2}{4}\int_0^t(\delta-bZ_s) ds\nonumber \\ 
&&+(2p-1)\int_0^t\mathbbm{I}_{\{\lambda_{J(y)}(s)>0\}}
2\sqrt{Z_s}\,d\ell_s^0(\sqrt{Z}-\lambda_{J(y)}), \\ \nonumber
\end{eqnarray}
and $Z_0=(H(y)+\gamma(J(y)))^2$ $Q_y$-a.s.\\
\begin{rem}\label{bemerk}
Setting $b=0$ , and replacing $\sigma^2 b\gamma_{J(y)}(s)$ by $\sigma^2 c$, $c\in \R^+$, 
in the expression for $Q_y$, 
we construct instead of (\ref{CIR}) a positive solution $Z=Y^2$ to 
\begin{eqnarray}\label{DSR} 
Z_t & = & Z_0+\sigma\int_0^t\sqrt{Z_s}dW_s+\frac{\sigma^2}{4}\int_0^t(\delta-c\sqrt{Z_s}) ds\nonumber \\ 
&&+(2p-1)\int_0^t\mathbbm{I}_{\{\lambda_{J(y)}(s)>0\}}
2\sqrt{Z_s}\,d\ell_s^0(\sqrt{Z}-\lambda_{J(y)}). \\ \nonumber 
\end{eqnarray}
For $p=\frac12$, (\ref{DSR}) is well-known as the double square-root (DSR) model of Longstaff 
in financial mathematics. For $p=\frac12$, (\ref{CIR}) is the well-known Cox-Ingersoll-Ross model (CIR).
\end{rem}
\begin{rem}\label{quasi}
The equations (\ref{rootCIR}), (\ref{CIR}), (\ref{DSR}), are in the sense of equivalence of additive 
functionals. This means that they hold for initial conditions outside some exceptional set. If $\lambda^2$ 
is constant, say $\lambda^2=c>0$, we are in the symmetric case, and it is clear that the parabolic capacity is comparable with 
the elliptic one, so that we can 
start from every $X_0,Y_0,Z_0=x\ge 0$, if $\delta<2$, and for every $X_0,Y_0,Z_0=x>0$, if $\delta\ge 2$. 
Indeed $Y$ is then associated to the symmetric (!) Dirichlet form (cf. \cite{fot}), which is uniquely determined as the closure of
$$
{\cal E}^p(f,g):=\int_0^{\infty}\frac{\sigma^2}{2}xf'(x)g'(x)x^{\frac{\delta}{2}-1}e^{-\frac{bx}{2}}\rho(x)dx;\ \ 
f,g\in C_0^{\infty}([0,\infty))
$$
in $L^2([0,\infty),x^{\frac{\delta}{2}}e^{-\frac{bx}{2}}\rho(x)dx)$, 
where $\rho(x)=(1-p)\mathbbm{I}_{\{x<c\}}+p\mathbbm{I}_{\{x\ge c\}}$, and it is clear that the capacities are all 
equivalent for $p\in (0,1)$. $p=\frac12$ corresponds to case without skew reflection.\\
The regularity of the equations (\ref{rootCIR}), (\ref{CIR}), (\ref{DSR}), i.e. the question whether we can start pointwise in the parabolic case, i.e. if $\lambda^2\not=const$ is subject of forthcoming work. 
Note however, that the structure of 
(\ref{rootCIR}), (\ref{CIR}), (\ref{DSR}), is not influenced by these questions of regularity. 
\end{rem}

\begin{rem}\label{whysym}
(i) Rewriting (\ref{fuku}) with $\overline{sgn}, \underline{sgn}$, one can easily see that 
 $\ell_t^{0+}(X-\beta)=p\ell_t^{\beta}$, and 
$\ell_t^{0-}(X-\beta)=(1-p)\int_0^t\mathbbm{I}_{\{\beta(s)>-\gamma(s)\}}d\ell_s^{\beta}$.
This implies the following relations for $\ell^0_t(R-\lambda^2)$ in (\ref{skewbesq}):
\begin{eqnarray}\label{relations}
\ell^0_t(R-\lambda^2)=\frac{1}{2p}\ell^{0+}_t(R-\lambda^2)=\frac{1}{2(1-p)} \ell^{0-}_t(R-\lambda^2).
\end{eqnarray}
Note that a stronger version of (\ref{relations}) holding for more general $\lambda^2$  is obtained in \cite{Tr6} by purely probabilistic methods.
One observes immediately the discontinuity of the local times in the space variable, thus we provide another example of diffusion 
with discontinuous local time (see e.g. \cite{walsh}). Moreover, if $|p|>1$, then any of these local times is identically zero. 
Consequently, the associated process is the CIR process, for which uniqueness in any sense is known to hold. 
Consider the time dependent Dirichlet form $\e$ (cf e.g. \cite{o2005}, \cite{Tr4}) corresponding to the time dependent CIR process $(t,R_t)$, then 
$$
{\cal E}(F,G):=\int_0^{\infty}\int_0^{\infty}
\frac{\sigma^2}{2}|x|\partial_x F(t,x)\partial_x G(t,x)|x|^{\frac{\delta}{2}-1}
e^{-\frac{bx}{2}}dxdt
$$
$$
- \int_0^{\infty}\int_0^{\infty}\partial_t F(t,x) G(t,x)|x|^{\frac{\delta}{2}-1}
e^{-\frac{bx}{2}}dxdt.\ \ 
$$
One can easily see that the local time $\ell^0(R-\lambda^2)$ is uniquely associated to the measure 
\begin{eqnarray*}
\frac{\sigma^2}{2}|x|^{\frac{\delta}{2}}e^{-\frac{bx}{2}}\delta_{\lambda^2(t)}(dx)dt.
\end{eqnarray*} 
But this measure doesn't vanish if $\lambda^2$ is different from zero on a set of 
positive Lebesgue measure. Therefore $\ell^0(R-\lambda^2)$ cannot vanish identically, a contradiction. Thus an associated process cannot exist if $|p|>1$.\\
(ii) The first reason to employ point symmetric derivatives, 
is that these better work out the intuitive structure of skew 
reflection, namely $2p-1=p-(1-p)$, so upper reflection 
with probability  $p$, and lower reflection with probability $1-p$. 
At least in the case ($\delta=1, \sigma=2, b=0, \lambda^2\equiv 0$) the just mentioned intuitive structure is rigorously described for the square root process (the skew BM) through excursion theory 
(see e.g. \cite{hs}, \cite{walsh}). The second reason is that symmetric local times correspond to symmetric derivatives, which are used in distribution theory, and therefore correspond to our analytic construction of the Markov process generator.    \\
\end{rem}
\begin{rem}\label{addcalc}
(i) The relation (\ref{relloc}) can easily be derived by writing down Fukushima's extended decomposition 
(\ref{I5}) in localized form for $|Z_t-\lambda^2(t)|=|(X+\gamma(t))^2-\lambda^2(t)|$ 
($Z$ as in (\ref{CIR})) and then comparing it with the symmetric Tanaka formula. More precisely, we obtain
$$
2\sqrt{Z_s}\,d\ell_s^0(\sqrt{Z}-\lambda_{J(y)})=d\ell_s^0(Z-\lambda^2_{J(y)})
$$
by doing so.\\ 
(ii) Coming back to Remark \ref{constloc}(ii), suppose that we had assumed $\lambda=\beta+\gamma$, where only 
$\gamma\in H^{1,2}_{loc}(\R^+)$, and $\beta$ not necessarily continuous, but decreasing. Then, we 
would have obtained 
(\ref{CIRroot}), (\ref{CIR}), except that $\ell^0(\sqrt{Z}-\lambda_{J(y)})$ has to be replaced by 
$\ell^{\lambda}$, where $\ell^{\lambda}$ is a positive continuous additive functional of $Y$, which only grows when $Y=\lambda$, or equivalently $Z=\lambda^2$.
\end{rem}

\begin{rem}[The case $\delta\in (0,1)$]\label{extension}
Since we are no longer in the semimartingale 
case we can no longer make use of the Girsanov formula with $\gamma$. Therefore one puts $\gamma\equiv 0$.
Then one may still start as before with (\ref{DF}) and the following density 
$$
\rho(t,x)=\left ((1-p) \mathbbm{I}_{[0,\lambda(t))}(x)+
p \mathbbm{I}_{[\lambda(t),\infty)}(x)\right )
|x|^{\delta-1}e^{-\frac{bx^2}{2}}.
$$ 
Note that $\rho(\cdot,x)$ is still assumed to be increasing in $t$, and that Proposition 
\ref{smoo} could no longer be available.\\
One may preferably directly start with the squared process since the technique of changing the measure 
will not be used. Thus one could proceed with the following (cf. (\ref{form})) time dependent form
\begin{eqnarray*}
\e(F,G) & = & \int_{0}^{\infty}\int_{0}^{\infty}\frac{\sigma^2}{2}|x| 
\partial_x F(s,x) \,\partial_x G(s,x) \rho(s,x) dxds\nonumber  \\
& & \hspace*{+2cm}-\int_{0}^{\infty}\int_{0}^{\infty}\partial_t F (s,x) G(s,x)\rho(s,x)dxds, 
\end{eqnarray*}
with 
$$
\rho(t,x)=\left ((1-p) \mathbbm{I}_{[0,\lambda^2(t))}(x)+
p \mathbbm{I}_{[\lambda^2(t),\infty)}(x)\right )
|x|^{\frac{\delta}{2}-1}e^{-\frac{bx}{2}},
$$ 
and where $p\in (0,1)$, and $\lambda^2\in H^{1,1}_{loc}(\R^+)$, are chosen, such that $\rho$ is increasing in $t$.
In order to convince the reader, we just line out the argument for the existence of the corresponding local time on $\lambda^2$. In fact one only has to show that 
$$
\frac{\sigma^2}{2}|x|^{\frac{\delta}{2}}e^{-\frac{bx}{2}}\delta_{\lambda^2(t)}(dx)dt
$$
is smooth with respect to the symmetric part of $\e$. This can be done analogously to 
Proposition \ref{smoo}. Of course, if $\lambda^2$ is a constant, we consider the 
 symmetric Dirichlet form of Remark \ref{quasi}.

\end{rem}

\end{document}